\documentclass[a4paper]{amsart}
\usepackage[latin1]{inputenc}
\usepackage{amssymb}
\usepackage{amsmath}
\usepackage{mathrsfs}
\usepackage{eufrak}
\usepackage{amsthm}
\usepackage{amsfonts}
\usepackage{textcomp}
\usepackage{graphicx}
\usepackage[pdftex]{color}
\usepackage{bm}
\usepackage{paralist}
\usepackage[shortlabels]{enumitem}
\usepackage{hyperref}
\usepackage{comment}
\usepackage{tikz}
\usetikzlibrary{calc,shadings,patterns}
\usetikzlibrary{matrix}
\usetikzlibrary{arrows}
\usetikzlibrary{matrix}

\usepackage{amsmath,amssymb}
\usepackage{eufrak}
\usepackage{amsthm}
\usepackage{enumitem}
\usepackage{setspace}
\usepackage{graphicx}
\usepackage{graphics}
\usepackage{color}
\usepackage{xcolor}
\usepackage{hyperref}
\usepackage{youngtab}
\usepackage{marginnote}
\usepackage{lscape}
\usepackage{faktor}
\usepackage{hyperref}
\usepackage[normalem]{ulem}
\usepackage{bm}
\usepackage{afterpage}
\usepackage{mathrsfs}
\usepackage[all]{xy}
\xyoption{matrix}
\xyoption{arrow}
\usepackage{tikz}
\usetikzlibrary{calc,shadings,patterns}
\usetikzlibrary{matrix}
\usetikzlibrary{arrows}
\usetikzlibrary{matrix}
\usepackage{verbatim}
\usepackage{multirow}

\newtheorem*{corollary*}{Corollary}
\newtheorem*{theorem*}{Theorem}
\newtheorem{theorem}{Theorem}[section]

\newtheorem{corollary}[theorem]{Corollary}
\newtheorem{lemma}[theorem]{Lemma}
\newtheorem{proposition}[theorem]{Proposition}
\newtheorem{question}[theorem]{Question}

\newtheorem*{claim*}{Claim}

\newtheorem*{conjecture 1}{Conjecture 1}
\newtheorem*{conjecture 2}{Conjecture 2}

\theoremstyle{definition}
\newtheorem{definition}[theorem]{Definition}

\newtheorem*{theorem }{Theorem}
\newtheorem{remark}[theorem]{Remark}
\newtheorem{example}[theorem]{Example}

\theoremstyle{remark}

\numberwithin{equation}{theorem}

\makeatletter
\renewcommand*\env@matrix[1][\
arraystretch]{%
  \edef\arraystretch{#1}%
  \hskip -\arraycolsep
  \let\@ifnextchar\new@ifnextchar
  \array{*\c@MaxMatrixCols c}}
\makeatother

\newcommand{\End}{\operatorname{End}}
\newcommand{\Hom}{\operatorname{Hom}}

\newcommand{\soc}{\operatorname{\mathrm{soc}}}

\newcommand{\cB}{{\mathcal B}}

\newcommand{\cS}{{\mathcal S}}

\newcommand{\pdim}{\operatorname{pdim}}
\newcommand{\gldim}{\operatorname{gldim}}

\begin{document}

\title{A new characterisation of quasi-hereditary Nakayama algebras and applications}
\date{\today}



\author{Ren\'{e} Marczinzik}
%\address{Institute of algebra and number theory, University of Stuttgart, Pfaffenwaldring 57, 70569 Stuttgart, Germany}
\email{marczire@mathematik.uni-stuttgart.de}

\author{Emre Sen}
%\address{Department of Mathematics, University of Iowa, Iowa City, IA}
\email{emresen641@gmail.com}

\begin{abstract}
We call a finite dimensional algebra $A$ \emph{S-connected} if the projective dimensions of the simple $A$-modules form an interval. We prove that a Nakayama algebra $A$ is S-connected if and only if $A$ is quasi-hereditary.
We apply this result to improve an inequality for the global dimension of quasi-hereditary Nakayama algebras due to Brown. We furthermore classify the Nakayama algebras where equality is attained in Brown's inequality and show that they are enumerated by the even indexed Fibonacci numbers if the algebra is cyclic and by the odd indexed Fibonacci numbers if the algebra is linear.
\end{abstract}

\maketitle
\section*{Introduction}
We assume that all algebras are finite dimensional, non-semisimple and connected over an algebraically closed field $K$. We call an algebra of finite global dimension \emph{S-connected} if the projective dimensions of the simple modules form an interval, that is we have that for every natural number $k$ with $\pdim S_1 \leq k \leq \pdim S_2$ for two simple $A$-modules $S_1$ and $S_2$ we also have $k=\pdim S_3$ for some simple $A$-module $S_3$. A \emph{Nakayama algebra} is an algebra such that every indecomposable module has a unique composition series. Nakayama algebras are one of the most fundamental classes of finite dimensional algebras with several recent connections to other areas such as cluster tilting theory \cite{JK} and combinatorics \cite{MRS}. We then prove:
\begin{theorem*}
A Nakayama algebra $A$ is S-connected if and only if $A$ is quasi-hereditary.
\end{theorem*}

In \cite{Bro} Brown proved that for a quasi-hereditary Nakayama algebra with $n$ simple modules there is the inequality $\gldim A \leq \lambda(A)+1 \leq n$, where $\lambda(A)$ is the number of simple $A$-modules with projective dimension not equal to 1. For linear Nakayama algebras Brown also proved the even stronger inequality $\gldim A \leq \lambda(A)$. Brown's inequality was also proven by Koenig in \cite{Koe} using exact Borel subalgebras for quasi-hereditary algebras. In this note we will see that this inequality is a special case of an inequality that holds for any finite dimensional algebra that is S-connected. For a finite dimensional algebra $A$, set $\mathcal{O}_A:= \{ pd(S) \mid S$ simple $\}$. For a value $c \in \mathcal{O}_A$ we define $\lambda_c(A):= |\{ i \mid pd(S_i) \neq c \}|$, the number of simple $A$-modules with projective dimension not equal to $c$. The generalization of Brown's result looks as follows:
\begin{theorem*} \label{inequ}
Let $A$ be an S-connected algebra with $\min(\mathcal{O}_A)=a$. Then $\gldim(A) \leq a+ \lambda_c(A)$ for any $c \in \mathcal{O}_A$. 
\end{theorem*}

For quasi-hereditary Nakayama algebras we have $\min(\mathcal{O}_A)=0$ when $A$ has an acyclic quiver and $\min(\mathcal{O}_A)=1$ when $A$ has a non-acyclic quiver and thus we obtain a generalization of Brown's result with any $c \in \mathcal{O}_A$, where choosing $c=1$ gives Brown's original result with $\lambda_c(A)=\lambda(A)$.

Our last main result gives a combinatorial classification of the quasi-hereditary Nakayama algebras where equality is attained in Brown's inequality $\gldim A \leq \lambda(A)+1$ when $A$ has a cyclic quiver and the inequality  $\gldim A \leq \lambda(A)$ when $A$ has a linear quiver. 
\begin{theorem*}
Let $n \geq 1$.
\begin{enumerate}
\item The number of connected cyclic quasi-hereditary Nakayama algebras with $n+1$ simple modules and $\gldim A = \lambda(A)+1$ is given by the even Fibonacci numbers $F_{2n}$.
\item The number of connected linear quasi-hereditary  Nakayama algebras with $n+1$ simple modules and $\gldim A = \lambda(A)$ is given by the odd Fibonacci numbers $F_{2n-1}$.
\end{enumerate}
\end{theorem*}
In the last section we give an outlook on other classes of algebras with regards to the property of being S-connected and we also give an example of a quasi-hereditary algebra that is not S-connected.
\section{Preliminaries}
We assume that all algebras are finite dimensional, non-semisimple and connected over an algebraically closed field $K$ and modules are right modules unless stated otherwise. We assume that the reader is familiar with the representation theory and homological algebra of finite dimensional algebras and refer for example to the book \cite{SkoYam}. $J$ denotes the Jacobson radical of an algebra $A$ and $D:=\Hom_A(-,K)$ the natural duality. Since we work over an algebraically closed field, every finite dimensional algebra $A$ is Morita equivalent to a quiver algebra $KQ/I$ and since all our notions are invariant under Morita equivalence, we can assume that our algebras are given by quiver and relations.
Let $\{e_1,...,e_n \}$ be a complete set of primitive orthogonal idempotents.
We denote the simple $A$-modules corresponding to the primitive idempotent $e_i$ by $S_i=e_i A/e_i J$, the indecomposable projective $A$-modules corresponding to $e_i$ by $P_i=e_i A$ and the indecomposable injective $A$-modules corresponding to $e_i$ by $I_i=D(Ae_i)$. When $M$ is a module, we denote by $P(M)$ the projective cover of $M$.
A \emph{Nakayama algebra} $A$ is an algebra such that every indecomposable $A$-module is uniserial, which means it has a unique composition series. It is well known that a Nakayama algebra $A$ has an acyclic quiver if and only if $A$ has a simple module of projective dimension zero.
We call a Nakayama algebra $A$ linear if its quiver is acyclic and we call $A$ cyclic if its quiver is not acyclic.
The \emph{Kupisch series} of a Nakayama algebra with $n$ simple modules is the sequence $[c_1,...,c_n]$ with $c_i = \dim(e_i A)$ that uniquely determines the Nakayama algebra up to isomorphism. 
The Kupisch series is constant (that is $c_i=c_j$ for all $i,j$) if and only if $A$ is selfinjective.
When $A$ is cyclic and not selfinjective, the Kupisch series is determined up to a cyclic shift by the conditions $c_{r}=c_{r+1}+1$ for some $r \in \{1,...,n\}$, $c_i -1 \leq c_{i+1} $ for all  $i \in \{1,...,n \}$ and $c_i \geq 2$ for all $i$ where we set $c_{n+1}=c_1$.

The quiver of a Nakayama algebra with cyclic quiver:

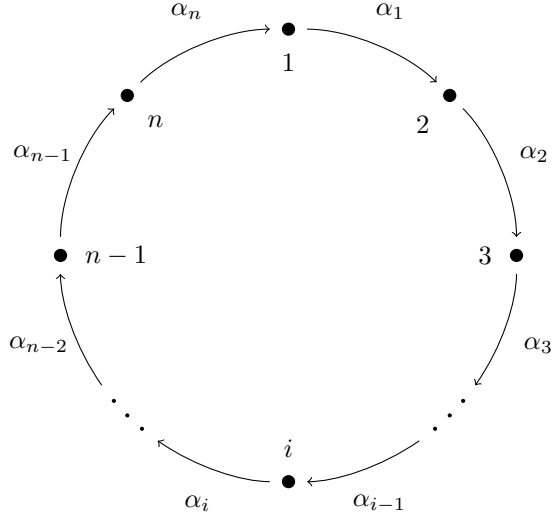
\begin{center}
\begin{tikzpicture}
% All nodes, node labels, and loops
\foreach \ang\lab\anch in {90/1/north, 45/2/{north east}, 0/3/east, 270/i/south, 180/{n-1}/west, 135/n/{north west}}{
  \draw[fill=black] ($(0,0)+(\ang:3)$) circle (.08);
  \node[anchor=\anch] at ($(0,0)+(\ang:2.8)$) {$\lab$};
  %\draw[->,shorten <=7pt, shorten >=7pt] ($(0,0)+(\ang:3)$).. controls +(\ang+40:1.5) and +(\ang-40:1.5) .. ($(0,0)+(\ang:3)$);
}

% Top part of circle, arrows between different nodes and their labels
\foreach \ang\lab in {90/1,45/2,180/{n-1},135/n}{
  \draw[->,shorten <=7pt, shorten >=7pt] ($(0,0)+(\ang:3)$) arc (\ang:\ang-45:3);
  \node at ($(0,0)+(\ang-22.5:3.5)$) {$\alpha_{\lab}$};
}

% Bottom part of circle, arrows between different nodes and their labels
\draw[->,shorten <=7pt] ($(0,0)+(0:3)$) arc (360:325:3);
\draw[->,shorten >=7pt] ($(0,0)+(305:3)$) arc (305:270:3);
\draw[->,shorten <=7pt] ($(0,0)+(270:3)$) arc (270:235:3);
\draw[->,shorten >=7pt] ($(0,0)+(215:3)$) arc (215:180:3);
\node at ($(0,0)+(0-20:3.5)$) {$\alpha_3$};
\node at ($(0,0)+(315-25:3.5)$) {$\alpha_{i-1}$};
\node at ($(0,0)+(270-20:3.5)$) {$\alpha_i$};
\node at ($(0,0)+(225-25:3.5)$) {$\alpha_{n-2}$};

% Ellipsis
\foreach \ang in {310,315,320,220,225,230}{
 \draw[fill=black] ($(0,0)+(\ang:3)$) circle (.02);
}
\end{tikzpicture}

\end{center}

The quiver of a Nakayama algebra with acyclic quiver:
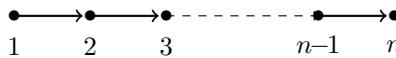
\begin{center}
\begin{center}
\begin{tikzpicture}
\draw[thick,->] (0,0) node{$\bullet$} --(0.9,0); 
\draw[thick,->] (1,0) node{$\bullet$} --(1.9,0); 
\draw[dashed,->] (2,0) node{$\bullet$} --(4,0); 
\draw[thick,->] (4,0) node{$\bullet$} --(4.9,0); 
\draw[thick] (5,0) node{$\bullet$} ;
\draw (0,-0.4) node{$1$} ; 
\draw (1,-0.4) node{$2$} ;
\draw (2,-0.4) node{$3$} ;
\draw (4,-0.4) node{$n\!\!-\!\!1$} ;
\draw (5,-0.4) node{$n$} ;
\end{tikzpicture}
\end{center}
\end{center}

An algebra $A$ is called \emph{quasi-hereditary} over a poset $X$ if there is a bijection from $X$ to the set of isomorphism classes of simple $A$-modules such that there exist for each $x \in X$ quotient modules $\Delta(x)$ of $P_x$ such that the following two conditions are satisfied:
\begin{enumerate}
\item The kernel of the canonical epimorphism $\Delta(x) \rightarrow S_x$ is filtered by simple modules $S_y$ with $y < x$.
\item The kernel of the canonical epimorphism $P_x \rightarrow \Delta(x)$ is filtered by $\Delta(t)$ with $t >x$.
\end{enumerate}
Every quasi-hereditary algebra has finite global dimension and every quiver algebra with an acyclic quiver is quasi-hereditary.
We refer for example to \cite{DR} for more on quasi-hereditary algebras.
We will need the following characterization of quasi-hereditary Nakayama algebras:
\begin{theorem}\cite[Proposition 3.1]{UY} \label{UYtheorem}
A Nakayama algebra $A$ is quasi-hereditary if and only if $A$ has a simple module of projective dimension zero or a simple module of projective dimension two.
\end{theorem}

\subsection{Syzygy Filtrations}
In this work, we use the syzygy filtration method introduced in \cite{sen2019syz}. For details, we refer to the papers \cite{sen2018varphi}, \cite{sen2019syz} and the appendix on the syzygy filtration method written by C.M. Ringel in \cite{ringel2020}.

Here, it is convenient for us to use an irredundant system of relations defining cyclic Nakayama algebras. We consider the irredundant system of relations $\alpha_{k_{2i}}\ldots\alpha_{k_{2i-1}}=0$ where $1\leq i\leq  r$ and $k_{f}\in\left\{1,2,\ldots,n\right\}$ for a cyclic oriented quiver $Q$ where each arrow $\alpha_i$, $1\leq i\leq n-1$ starts at the vertex $i$ and ends at the vertex $i+1$ and $\alpha_n$ starts at vertex $n$ and ends at vertex $1$. Let $I$ be an admissible ideal generated by the relations:
\begin{gather}\label{relations}
\alpha_{k_2}\ldots\alpha_{k_1+1}\alpha_{k_1}\ \ =0 \\
\alpha_{k_4}\ldots\alpha_{k_3+1}\alpha_{k_3}\ \ =0\nonumber \\
%\alpha_{k_6}\ldots\alpha_{k_5+1}\alpha_{k_5}\ \ =0\nonumber\\
\vdots \nonumber\\
\alpha_{k_{2r-2}}\ldots\alpha_{k_{2r-3}+1}\alpha_{k_{2r-3}}=0\nonumber\\
\alpha_{k_{2r}}\ldots\alpha_{k_{2r-1}+1}\alpha_{k_{2r-1}}=0\nonumber
\end{gather}
where $..<k_1<k_3<\ldots<k_{2r-1}<k_1<\ldots$ is cyclically ordered \cite{sen2018varphi}. Then the bound quiver algebra $kQ/I$ is a cyclic Nakayama algebra.

\begin{definition}
Let $\cS(A)$ be the complete set of representatives of the socles of projective modules over $A$ i.e. $\cS(A)=\left\{S_{k_2}, S_{k_4},\ldots,S_{k_{2r}}\right\}$ and let $\cS'(A)$ be the complete set of representatives of simple modules such that they are indexed by one cyclically larger indices of $\cS(A)$ i.e. $\cS'(A)=\left\{S_{k_{2}+1}, S_{k_4+1},\ldots,S_{k_{2r}+1}\right\}$. We define the base set $\cB(A)$ as:

\begin{gather}
\cB(A):=\{ \Delta_1\cong\begin{vmatrix}[1]
    S_{k_{2r}+1} \\
    \vdots  \\
    S_{k_{2}}
\end{vmatrix}\!, \Delta_2\cong\begin{vmatrix}[1]
    S_{k_{2}+1}  \\
    \vdots  \\
   S_{k_{4}}
\end{vmatrix}\!,\ldots,\Delta_j\cong\begin{vmatrix}[1]
   S_{k_{2(j-1)}+1}  \\
    \vdots  \\
    S_{k_{2j}}
\end{vmatrix}\!,..,\Delta_r\cong \begin{vmatrix}[1]
   S_{k_{2r-2}+1}  \\
    \vdots  \\
    S_{k_{2r}}
\end{vmatrix}\! \}
\end{gather}

\end{definition}
According to C.M. Ringel \cite{ringel2020} Appendix C, $\cB(A)$ modules can be realized as the first syzygy modules of the valley modules i.e. longest non-projective radicals of projective modules.

\begin{definition}\label{filteredalg} \cite{sen2019syz}
Let $A$ be a cyclic Nakayama algebra. The \emph{syzygy filtered algebra} $\bm{\varepsilon}(A)$ is:
\begin{align}
\bm{\varepsilon}(A):=\End_{A}\left(\bigoplus\limits_{S\in \cS'(A)}P(S)\right)
\end{align} 
The $d$th syzygy filtered algebra $\bm{\varepsilon}^d(A)$ is:
\begin{align}
\bm{\varepsilon}^d(A):=\End_{\bm{\varepsilon}^{d-1}(A)}\left(\bigoplus\limits_{S\in\cS'(\bm{\varepsilon}^{d-1}(A))}P(S)\right)
\end{align}
provided that $\bm{\varepsilon}^{d-1}(A)$ is a cyclic non-selfinjective Nakayama algebra.
\end{definition}

\begin{remark}\label{remarkfacts} We collect and summarize some useful results about the syzygy filtration method.
\begin{enumerate}[label=\roman*.]
\item\label{listfiltration}  The second and higher syzygies of $A$-modules have unique $\cB(A)$ filtrations \cite{sen2018varphi}.
\item\label{listcategory} The category of $\cB(A)$ filtered $A$-modules is equivalent to the category of ${\bm\varepsilon}(A)$-modules \cite{sen2019syz}.
\item\label{listnakayama} ${\bm\varepsilon}(A)$ is a Nakayama algebra \cite{sen2019syz}.
\item\label{listreduction} The $\bm\varepsilon$-construction reduces the following homological dimensions by exactly two: $\varphi$-dimension, finitistic, dominant, Gorenstein dimensions and the delooping level \cite{sen2019syz}, \cite{sen2020del}.
\item\label{listreduction2} If the global dimension of $A$ is infinite then there exists $d$ such that ${\bm\varepsilon}^d(A)$ is a selfinjective Nakayama algebra. \cite{sen2019syz}
\item\label{remarklis1} If the global dimension of $A$ is finite, then there exists a number $m$ such that ${\bm\varepsilon}^m(A)$ is a cyclic Nakayama algebra and ${\bm\varepsilon}^{m+1}(A)$ is a linear Nakayama algebra.\cite{sen2019syz}
\end{enumerate}
\end{remark}

\begin{remark} By using the syzygy filtration method, Theorem \ref{UYtheorem} can be restated as: a Nakayama algebra $A$ is quasihereditary if and only if either $A$ or $\bm\varepsilon(A)$ is linear Nakayama algebra.
\end{remark}

We will also need the following result of Madsen for Nakayama algebras:

\begin{proposition}\cite[Proposition 5.1 (including the proof)]{Mad} \label{madsenresult}
Let $A$ be a Nakayama algebra and $M$ an indecomposable $A$-module with odd projective dimension.
Then 
$$\pdim M= \sup \{ \pdim S | S \ \text{is a simple composition factor of} \ M \}.$$

\end{proposition}
$\mathbb{N}_0$ denote the natural numbers including 0.
For two natural numbers $a,b$ we denote by $[a,b]:=\{k \in \mathbb{N}_0 \mid a \leq k \leq b \}$ the interval between $a$ and $b$ in $\mathbb{N}_0$.

\section{S-connected algebras and quasi-hereditary Nakayama algebras}
For a finite dimensional algebra $A$, we define $\mathcal{O}_A:= \{ \pdim S | S$ simple $\}$. We say that an algebra $A$ of finite global dimension is \emph{S-connected} if $\mathcal{O}_A$ is an interval, that is we have that for every natural number $k$ with $\pdim S_1 \leq k \leq \pdim S_2$ for two simple $A$-modules $S_1 , S_2$ we also have $k=\pdim S_3$ for some simple $A$-module $S_3$. Not every algebra of finite global dimension is S-connected as the next examples shows:
\begin{example} \label{examplenak}
Let $A$ be the Nakayama algebra with Kupisch series [3,4,4]. Then $A$ has global dimension 4 and the simple $A$-modules have projective dimension 1,3 and 4 and thus $A$ is not S-connected.

\end{example}

However, the class of S-connected algebras includes many algebras, for example any algebra with an acyclic quiver as the next proposition shows:
\begin{proposition} \label{acyclicpropo}
Let $A=KQ/I$ be a quiver algebra with acyclic quiver $Q$. Then $A$ is S-connected.

\end{proposition}
\begin{proof}
We use induction on the number of points of $Q$ or equivalently the number of simple $A$-modules. The statement is certainly true when the quiver $Q$ has exactly one point. Now assume that $A$ has $n \geq 2$ simple modules and that the statement is true for all algebras with an acyclic quiver and with at most $n-1$ simple modules.
Let $v$ be a source vertex of $Q$, and $B$ the quiver algebra obtained by removing $v$. By induction, $B$ has simple modules of all projective dimensions up to $\gldim B$. 

The simple modules for $B$ have the same projective dimensions as the corresponding simple modules for $A$, and the remaining simple module $S$ corresponding to the vertex $v$ in $A$ has projective dimension at most $\text{gldim} B+1$. To see this, assume that $S$ has projective dimension strictly larger than $\gldim B +1$. Then the module $\Omega^1(S)$ has projective dimension at least $\gldim B +1$ and is also a $B$-module, which is a contradiction.

By induction $\mathcal{O}_B=[a,b]$ is an interval and we saw that $\mathcal{O}_A=\mathcal{O}_B$ or $\mathcal{O}_A=[a,b+1]$ and in both cases $\mathcal{O}_A$ is an interval and the statement is proven. \qedhere

\end{proof}

\begin{corollary} \label{linear}
Let $A$ be a linear Nakayama algebra with global dimension $g$. Then $\mathcal{O}_A=[0,g]$ and $A$ is S-connected.
\end{corollary}

We now study properties of $\mathcal{O}_A$ when $A$ is a general Nakayama algebra of finite global dimension.

\begin{proposition} \label{result 1}
Let $A$ be a connected Nakayama algebra with finite global dimension $g$.
For every odd number $l$ with $1 \leq l \leq g$ there exists a simple module with projective dimension $l$.
\end{proposition}
\begin{comment}
\begin{proof}
The result is clear when $A$ is linear by \ref{linear} and thus we can assume that $A$ is a cyclic Nakayama algebra.
Assume that ${\bm\varepsilon}(A)$ is a linear Nakayama algebra. Therefore there are simple modules in ${\bm\varepsilon}(A)$ that attain the odd numbers in the interval: 
\begin{align}
1\leq \pdim_{\bm\varepsilon(A)}S\leq m
\end{align}
Therefore, in the algebra $A$ all odd numbers in the following interval are attained by some simple $A$-modules:
\begin{align}
3\leq \pdim_{A}S'\leq m+2
\end{align}
But, since $A$ is a cyclic Nakayama algebra, by the system of relations, there has to be simple module of projective dimension one. Hence, we extend the interval:
\begin{align}
1\leq \pdim_{A}S'\leq m+2
\end{align}
Assume that there exists algebra $A'$ which is syzygy equivalent to $A$, therefore we get:
\begin{align}
3\leq \pdim_{A'}S\leq m+4
\end{align}
By the similar argument, there has to be simple $A'$ module of projective dimension one. The result follows by induction.
\end{proof}
\end{comment}
\begin{proof}
The result is clear when $A$ is linear by \ref{linear}.
Clearly for each odd $l$ with $1 \leq l \leq g$ there exists an indecomposable module $M$ with projective dimension $l$ and by \ref{madsenresult} there also exists a simple composition factor of $M$ with projective dimension $l$.\end{proof}

The analogue result for even numbers is not true as we saw in the example \ref{examplenak}.
However, the following is still true:

\begin{proposition} \label{result 2}
Let $A$ be Nakayama algebra of finite global dimension with two simple modules $S_x,S_y$ satisfying $\pdim S_x=2m\leq\pdim S_y=2M$. Then for each number $t$ between $m$ and $M$ there exists a simple module $S_z$ satisfying $\pdim S_z=2t$.
\end{proposition}
\begin{proof}
The result is true when $A$ is linear by \ref{linear} and thus we can assume $A$ is a cyclic Nakayama algebra in the rest of the proof.
Since $A$ is of finite global dimension, by  remark \ref{remarkfacts} \ref{remarklis1}, there exists a minimal $k$ such that $\bm{\varepsilon}^{k}(A)$ is a \emph{linear} Nakayama algebra and $\bm{\varepsilon}^{k-1}(A)$ is a cyclic Nakayama algebra.
By proposition \ref{linear}, there exists simple $\bm{\varepsilon}^k(A)$-modules $S'_1,S'_2,\ldots,S'_p$ such that:
\begin{gather*}
\pdim_{\bm{\varepsilon}^k(A)} S'_1=2m'\\
\pdim_{\bm{\varepsilon}^k(A)} S'_2=2(m'+1)\\
\vdots\\
\pdim_{\bm{\varepsilon}^k(A)} S'_p=2M'
\end{gather*}
where $m'\leq M'$. But every $\bm{\varepsilon}^k(A)$-module can be obtained as a $2k^{th}$ syzygy of an $A$-module (remark \ref{remarkfacts} \ref{listreduction}). Now it is enough to show that these $A$-modules are actually simple $A$-modules. This follows from the construction of $\cB(A)$-modules in \cite{sen2019syz} (or $\Delta$-modules in \cite{sen2018varphi}), because the category of $\cB(A)$-filtered modules is equivalent to the category of ${\bm\varepsilon}(A)$ modules (remark \ref{remarkfacts}  \ref{listcategory}) and in particular simple ${\bm\varepsilon}(A)$ modules are equivalent to modules in $\cB(A)$. By induction (recall the construction of higher syzygy filtered algebras \ref{filteredalg}), every simple $\bm{\varepsilon}^k(A)$-module can be obtained as the $2k^{th}$ syzygy of simple $A$ modules. Without loss of generality, let $\Omega^{2k}(S_{x_i})\cong {}_{A}\Delta_i$ for $i$. We get:
\begin{gather*}
\pdim_{A} S_{x_1}=2k+\pdim_{A} \Delta_1=2k+\pdim_{\bm{\varepsilon}^k(A)}S'_1=2k+2m'\\
\pdim_{A} S_{x_2}=2k+\pdim_{A} \Delta_2=2k+\pdim_{\bm{\varepsilon}^k(A)}S'_2=2k+2(m'+1)\\
\vdots\\
\pdim_{A} S_{x_p}=2k+\pdim_{A} \Delta_p=2k+\pdim_{\bm{\varepsilon}^k(A)}S'_p=2k+2M'
\end{gather*}
Setting $m=k+m'$ and $M=k+M'$ finishes the proof.
\end{proof}
%{\color{red}\begin{remark} Every cyclic Nakayama algebra  has a simple module of projective dimension one. This fact together with the arguments in the above proof can be used to give another proof to proposition \ref{madsenresult}.
%\end{remark}}

We can now prove our first main result:

\begin{theorem} \label{Sconnectedquasihered}
Let $A$ be a Nakayama algebra. Then $A$ is S-connected if and only if $A$ is quasi-hereditary.
\end{theorem}
\begin{proof}
The result is clear when $A$ is linear, since then it is clearly quasi-hereditary and also S-connected by \ref{linear}. \newline 
Now assume that $A$ has a cyclic quiver for the rest of this proof.
We can assume that $A$ is not selfinjective and thus the Kupisch series is not constant, since selfinjective algebras are neither S-connected nor quasi-hereditary (since we assume that our algebras are not semi-simple).
Then the simple module $S_r=e_r A/ e_r J$ has projective dimension equal to one since $e_r J$ is projective when $c_{r}=c_{r+1}+1$.
Assume that $A$ is quasi-hereditary with global dimension $g$.
By \ref{UYtheorem} $A$ has a simple modules of projective dimension 2. Thus by \ref{result 1} and \ref{result 2} for every values $l$ with $1 \leq l \leq g$ there exists a simple module with projective dimension equal to $l$ and $A$ is S-connected. \newline
Now assume that $A$ is S-connected. Since $A$ is cyclic and thus not hereditary, it has global dimension $g \geq 2$ and we saw also that $A$ has a simple module of projective dimension 1. Since $A$ is S-connected and $2 \in [1,g]$, $A$ has a simple module of projective dimension two and thus is quasi-hereditary by \ref{UYtheorem}.
\end{proof}

We now look at inequalities for the global dimension of S-connected algebras. For a value $c \in \mathcal{O}_A$ we define $\lambda_c(A):= |\{ i \mid pd(S_i) \neq c \}|$, the number of simple $A$-modules with projective dimension not equal to $c$. 
\begin{theorem} \label{inequtheorem}
Let $A$ be an S-connected algebra with $\min(\mathcal{O}_A)=a$. Then $\gldim(A) \leq a+ \lambda_c(A)$ for any $c \in \mathcal{O}_A$. 

\end{theorem}
\begin{proof}
We have $\gldim A=a+(g-a)$ with $g= \gldim(A)$. There are $g-a+1$ distinct values for the projective dimensions of simple $A$-modules, namely $a,a+1,...,g-1,g$ since we assume that $A$ is S-connected. Thus the set $\{ i \mid  \pdim(S_i) \neq c \}$ has at least $g-a$ many different $i$ with $S_i=p$ such that $p \neq c$.
This shows $\gldim A = a+(g-a) \leq a+ |\{ i \mid  \pdim(S_i) \neq c \}|.$

\end{proof}

We record two corollaries of the previous theorem.

\begin{corollary}
Let $A$ be a quiver algebra with an acyclic quiver and $n$ simple modules. Then $A$ has global dimension at most $n-s \leq n-1$, where $s$ is the number of sinks in the quiver of $A$.

\end{corollary}
\begin{proof}
For a simple $A$-module $S_x$ we have $\pdim S_x=0$ if and only if $x$ is a sink.
Thus $\min(\mathcal{O}_A)=0$ since the quiver of $A$ is acyclic and therefore has sinks. 
Then the result follows from \ref{inequtheorem} by choosing $c=0$, since then $n-s$ is the number of non-sink vertices which corresponds to the simple $A$-modules with non-zero projective dimensions.

\end{proof}

The previous corollary gives a slight generalisation with a very easy proof of the well known result $\gldim A \leq n-1$ for acyclic quiver algebras $A$ with $n$ simple modules, see for example \cite{Farn}.

As an easy corollary we obtain from \ref{inequtheorem} the inequality of Brown for quasi-hereditary Nakayama algebras with a very easy proof:
\begin{corollary} \label{browninequcorollary}
Let $A$ be a quasi-hereditary Nakayama algebra. 
\begin{enumerate}
\item If $A$ is linear, $\gldim A \leq \lambda_1(A)$.
\item If $A$ is cyclic, $\gldim A \leq \lambda_1(A)+1$.

\end{enumerate}

\end{corollary}
\begin{proof}
This follows from \ref{inequtheorem} since $A$ is S-connected by \ref{Sconnectedquasihered} and $\min(\mathcal{O}_A)=0$ when $A$ is linear and $\min(\mathcal{O}_A)=1$ when $A$ is cyclic.
\end{proof}

For an inequality for the global dimension of general Nakayama algebras with finite global dimension and applications to higher Auslander algebras we refer to \cite{MM} and \cite{MMZ}. The full solution to the conjecture on higher Auslander algebras described in \cite{MMZ} can be found in \cite{sen2020aus}.

\section{Enumeration of quasi-hereditary Nakayama algebras with maximal global dimension}
Quasi-hereditary Nakayama algebras satisfy Brown's inequality $\gldim A \leq \lambda_1(A)+1$ as we saw in \ref{browninequcorollary}.
In this section we investigate quasi-hereditary Nakayama algebras satisfying Brown's inequality sharply, that is Nakayama algebras such that $\gldim A = \lambda_1(A)+1$. We remark that $\gldim A= \lambda_1(A)+1$ can not happen for Nakayama algebras with a linear quiver since there we have $\gldim A \leq \lambda_1(A)$ by \ref{browninequcorollary}. We analyze the case of linear Nakayama algebras satisfying $\gldim A=\lambda_1(A)$ in the section \ref{secLinear}.

We say an irredundant system of relations form a chain if the terminal points of the relations \ref{relations} satisfy:
\begin{align}\label{fibochaineq}
k_{2i-1}\leq k_{2i-2}<k_{2i+1}
\end{align}
for all $i\in [1,r-1]$, together with cyclic orderings:
\begin{align}
k_1<k_3<\cdots< k_{2r-1}\\
k_2<k_4<\cdots < k_{2r}
\end{align}.
\begin{example}\label{examplefibo1}
Let $n=4$, $r=2$. All possible chains and the corresponding Kupisch series are:

\begin{enumerate}
\item $\alpha_2\alpha_1=0$,$\alpha_3\alpha_2=0$, $[4,3,2,2]$
\item $\alpha_2\alpha_1=0$,$\alpha_4\alpha_3\alpha_2=0$, $[4,3,2,3]$
\item $\alpha_3\alpha_2\alpha_1=0$,$\alpha_4\alpha_3\alpha_2=0$, $[5,4,3,3]$
\item $\alpha_3\alpha_2\alpha_1=0$,$\alpha_4\alpha_3=0$, $[4,3,3,2]$
\end{enumerate} 
\end{example}
Two Kupisch series define the same algebra if two of them are equivalent under cyclic permutations, therefore we can choose $k_1=1$ and $k_{2r}\leq n$ in the level of relations.

\begin{remark} By the definition \ref{inequ}, $\lambda_1(A)=r$ where $r$ is the number of relations defining the algebra $A$.
\end{remark}
We set $r=\lambda_1(A)$ for the rest of this section.
\begin{theorem}\label{thmChain}
Let $A$ be cyclic Nakayama algebra with $n$ simple modules. The global dimension of $A$ is $r+1$ if and only if the defining relations of $A$ form a chain.
\end{theorem}
\begin{proof}
The only if part is obvious. Briefly, we assume that the relations form a chain. Let $M$ be a subquotient of the projective module $P_1$. We have
$\soc\Omega^{i}(M)\cong S_{k_{2i}}$ if $i\leq r$ and $\soc\Omega^{r+1}(M)\cong S_{k_2}$. Moreover $\Omega^{r+1}(M)$ is a projective module because $P_1$ is always submodule of it by the chain condition. Hence $\gldim M=r+1$.

For the other part, we start with $\gldim(A)=\lambda_1(A)+1=r+1$. This implies that the number of relations is $r$. Because the global dimension is $r+1$, there exists a module $M$ with the following projective resolution:
\begin{align}
0\longrightarrow P_{r+1}\longrightarrow\ldots \longrightarrow P_2\longrightarrow P_1\longrightarrow M \longrightarrow 0.
\end{align}
Here the subscripts correspond {\textbf{not}} to simple modules but just keep track of the position of the projective module in the resolution. The key observation is: two consecutive projective modules in the resolution have to come from two consecutive classes of projective modules. Otherwise: either we get a shorter projective dimension or infinite global dimension. In details: because the number of relations is $r$, the number of the classes of the projective modules is $r$. Without loss of generality, let the first relation start with one i.e. $M$ is a quotient of the projective module $P(S_1)$. Assume that
$\soc P_i\cong S_{k_{2j}}$ and $\soc P_{i+1}\cong S_{k_{2t}}\ncong S_{k_{2j+2}}$, then we have:
\begin{itemize}
\item either $j<t$ which makes projective dimension smaller than $r+1$ 
\item or $t<j$ which makes projective dimension infinite (because we get periodic syzygies).
\end{itemize}
Therefore from each the classes of projective modules we pick exactly one projective module in the cyclic ordering for the resolution. This only happens when the relations form the chain i.e. the indices of the socles of the syzygies are consecutive in the view of cyclic ordering.
\end{proof}

Now, we see that the number of algebras satisfying $\gldim(A)=r+1$ is related to one of the most famous integer sequence. Recall that \emph{Fibonacci numbers} $F_n$ are defined recursively by $F_0=0, F_1=1$ and $F_{n+1}=F_n+F_{n-1}$ for $n \geq 1$.
\begin{theorem}\label{thmfiboeven}
The number of cyclic Nakayama algebras with $n$ simple modules and global dimension $r+1$ is given by even indexed Fibonacci numbers.
\end{theorem}
\begin{proof}
By the theorem above, it is enough to count the all chains for a given $n$. By \cite{sen2018varphi}, $1\leq r\leq n-1$ if the algebra is not selfinjective. We define the function $E(n,r)$ which returns the nonisomorphic algebras with $n$ simple modules and number of relations $r$. The total number will be the sum of $E(n,r)$ over all possible $r$. By the propositions \ref{fiboprop1} and \ref{fibopro2} below, the claim follows. 
\end{proof}
 For instance, in the example \ref{examplefibo1}, $E(4,2)=4$.
\begin{proposition}\label{fiboprop1} 
$E(n,r)=\sum\limits^r_{i=1}{n-1\choose 2r-i}{r-1\choose i-1}$
\end{proposition}
\begin{proof}
For an arbitrary $r$, we need to count all integer solutions to the equations \ref{fibochaineq} i.e. satisfying the chain condition. This leads to a well known enumerative combinatorics problem. The number of solutions is ${n-1\choose 2r-1}$ if all inequalities are strict. Then there can be all strict inequalities but one, which leads to ${n-1\choose 2r-2}{r-1\choose 1}$. By similar counting arguments, the general term is ${n-1\choose 2r-i}{r-1\choose i-1}$, which gives the desired result.
\end{proof}
\begin{corollary}\label{fibocor1}
$E(n,r)={n+r-2\choose 2r-1}$
\end{corollary}
\begin{proof}
This follows using the formula ${n\choose a}+{n\choose a+1}={n+1\choose a+1}$ and induction.
\end{proof}

\begin{proposition}\label{fibopro2}
$F_{2n-2}=\sum\limits^{n-1}_{r=1} E(n,r)$.
\end{proposition}

\begin{proof}
We will prove it by induction. We set $F_{2n-2}=a_n$ and then the even Fibonacci sequence satisfies the recursion $a_{n+1}=3a_{n}-a_{n-1}$ and the first few elements are $0,1,3,8,21$, see for example \cite{OEIS}.
If $n=2$ then there can be only one relation, hence $a_2=E(2,1)=1$, which is the algebra $[3,2]$. If $n=3$, there are three possibilities either two distinct one relation cases $[4,3,2]$, $[5,4,3]$ or two relations $[3,2,2]$. Therefore $E(3,2)=1$, $E(3,1)=2$ which gives $3$.

We will calculate the recursion.
\begin{align}
a_n=\sum\limits^{n-1}_{r=1} {n+r-2\choose 2r-1}={n-1\choose 1}+{n\choose 3}+{n+1\choose 5}+\cdots+ {2n-3\choose 2n-3}&\\
a_{n+1}=\sum\limits^{n}_{r=1} {n+r-1\choose 2r-1}={n\choose 1}+{n+1\choose 3}+{n+2\choose 5}+\cdots+ {2n-2\choose 2n-3}&+{2n-1\choose 2n-1}
\end{align}
Their difference is:
\begin{align}\label{differenceEvenFibo}
a_{n+1}-a_{n}=&{n-1\choose 0}+{n\choose 2}+{n+1\choose 4}+\cdots+ {2n-3\choose 2n-4}+{2n-1\choose 2n-1}\\
2a_{n+1}-a_{n}=&{n-1\choose 0}+{n+1\choose 2}+{n+2\choose 4}+\cdots+ {2n-2\choose 2n-4}+{2n-2\choose 2n-3}+2{2n-1\choose 2n-1}\\
3a_{n+1}-a_{n}=&{n+1\choose 1}+{n+2\choose 3}+{n+3\choose 5}+\cdots+ {2n-1\choose 2n-3}+{2n-2\choose 2n-3}++3{2n-1\choose 2n-1}
\end{align}
The last two summands give:
\begin{align}
{2n-2\choose 2n-3}+3{2n-1\choose 2n-1}=2n-2+3=2n+1={2n\choose 2n-1}+{2n+1\choose 2n+1}
\end{align}
therefore:
\begin{align}
3a_{n+1}-a_{n}={n+1\choose 1}+{n+2\choose 3}+{n+3\choose 5}+\cdots+ {2n-1\choose 2n-3}+{2n\choose 2n-1}+{2n+1\choose 2n+1}=a_{n+2}.
\end{align}
This finishes proof. 
\end{proof}

\begin{example} If $n=3$, number of algebras $2+1$ follows from:
\begin{itemize}
\item $r=1$, $[4,3,2], [5,4,3]$ relations $\alpha_2\alpha_1$, $\alpha_3\alpha_2\alpha_1$
\item $r=2$ $[3,2,2]$ relations $\alpha_2\alpha_1, \alpha_3\alpha_2$
\end{itemize}
If $n=4$, we have $3+4+1$ via:
\begin{itemize}
\item $r=1$, $[5,4,3,2]$, $[6,5,4,3]$, $[7,6,5,4]$, relations are $\alpha_2\alpha_1$, $\alpha_3\alpha_2\alpha_1$, $\alpha_4\alpha_3\alpha_2\alpha_1$
\item $r=2$, see example \ref{examplefibo1}.
\item $r=3$, relations are $\alpha_2\alpha_1,\alpha_3\alpha_2,\alpha_4\alpha_3$, so Kupisch series is $[3,2,2,2]$
\end{itemize}
\end{example}

\section{Enumeration of Linear Nakayama algebras}\label{secLinear}
Let $A$ be a connected linear Nakayama algebra with $n$ simple modules. We can still describe the algebra in terms of relations similar to \ref{relations}, with one exception because there exists a simple module which is also projective. Because of this we always inculde the relation $\alpha_n=0$, which minimally produces the simple module $S_n$.

We fix the following set up. Let $I$ be an admissible ideal generated by the relations:
\begin{gather}\label{relations}
\alpha_{k_2}\ldots\alpha_{k_1+1}\alpha_{k_1}\ \ =0 \\
\alpha_{k_4}\ldots\alpha_{k_3+1}\alpha_{k_3}\ \ =0\nonumber \\
%\alpha_{k_6}\ldots\alpha_{k_5+1}\alpha_{k_5}\ \ =0\nonumber\\
\vdots \nonumber\\
\alpha_{k_{2r-2}}\ldots\alpha_{k_{2r-3}+1}\alpha_{k_{2r-3}}=0\nonumber\\
\alpha_{k_{2r}}\ldots\alpha_{k_{2r-1}+1}\alpha_{k_{2r-1}}=0\nonumber\\
\alpha_{n}=0
\end{gather}
where $1\leq k_1<k_3<\ldots<k_{2r-1}<k_{2r}<n$ for the linear quiver . Notice that $n$ has to be greater than $k_{2r}$ to make the relations irredundant. Note that we describe the simple projective module $S_n$ via the relation $\alpha_n$ to keep the construction consistent to  \ref{relations}.
\begin{lemma} If $A$ is connected linear Nakayama algebra, then $\lambda_1(A)$ gives the number of relations which define the algebra minimally.
\end{lemma}
\begin{proof}
All simple modules except $S_{k_1},S_{k_3},\ldots,S_{k_{2r-1}},S_n$ are of projective dimension one. Their number is equal to the number of relations.
\end{proof}

\begin{definition}\label{linearChain} We say that the defining relations of a {\bf linear} Nakayama algebra form a chain if they satisfy:
\begin{align}\label{linearEquations}
k_{2i-1}\leq k_{2i-2}<k_{2i+1}
\end{align}
for all $i\in [1,r-1]$, together with:
\begin{align}
1\leq k_1<k_3<\cdots< k_{2r-1}\\
k_2<k_4<\cdots < k_{2r}<n.
\end{align}
\end{definition}

\begin{proposition} Let $A$ be a connected linear Nakayama algebra with $n$ simple modules. The global dimension of $A$ is equal to $r$ if and only if the relations form a chain as in definition \ref{linearChain}.
\end{proposition}
\begin{proof}
We will follow similar arguments as in the proof of theorem  \ref{thmChain}. 
The only difference is the existence of the simple projetive module $S_n$.

First we handle the only if part. Assume that the relations form a chain. Let $M$ be a subquotient of the projective module $P_1$. We have
$\soc\Omega^{i}(M)\cong S_{k_{2i}}$ if $i\leq r$.  Moreover $\Omega^{r}(M)$ is a projective module because $S_n$ is always a submodule of it by the chain condition. Hence $\gldim M=r$.

For the other part, we start with $\gldim(A)=\lambda_1(A)=r$. This implies that the number of relations is $r$. Because the global dimension is $r$, there exists a module $M$ with the following projective resolution:
\begin{align}
0\longrightarrow P_{r}\longrightarrow\ldots \longrightarrow P_2\longrightarrow P_1\longrightarrow M \longrightarrow 0.
\end{align}
Here the subscripts correspond {\textbf{not}} to simple modules but just keep track of the position of the projective module in the resolution. The key observation is: two consecutive projective modules in the resolution have to come from two consecutive classes of projective modules. Otherwise we get a shorter projective dimension. In details: because the number of relations is $r$, the number of the classes of the projective modules is $r$. Let $M$ be a quotient of the projective module $P(S_1)$. Assume that
$\soc P_i\cong S_{k_{2j}}$ and $\soc P_{i+1}\cong S_{k_{2t}}\ncong S_{k_{2j+2}}$, then we have:
$j<t$ which makes the projective dimension smaller than $r$.
To reach the maximal global dimension, from each of the classes of projective modules we pick exactly one projective module. This is equivalent to the chain condition \ref{linearChain}.

\end{proof}

\begin{theorem} The number of connected linear Nakayama algebras with $n$ simple modules and global dimension $r$ is given by the odd indexed Fibonacci numbers.
\end{theorem}
\begin{proof}
We will follow similar arguments as in the proof \ref{thmfiboeven}. The only difference is again the existence of the simple projective module. Let $L(n,r)$ be the function which returns the nonisomorphic linear algebras with $n$ simple modules and number of relations $r$. We also include $\mathbb{A}_n$ type quivers, they corresponds to the $r=1$ case i.e $\alpha_n=0$.

We seek the integer solutions to the equations \ref{linearEquations}. 
By the same arguments as in the proof of proposition \ref{fiboprop1} and its corollary, we obtain
\begin{align}
L(n,r)=\sum\limits^{r-2}_{i=0}{n-1\choose 2r-i-2}{r-2\choose i}={n+r-3\choose 2r-2}
\end{align}

We claim that $\sum\limits^{n-1}_{r=1}L(n,r)=F_{2n-3}$. Notice that the sum:
\begin{align}
\sum\limits^{n-1}_{r=1}L(n,r)=\sum\limits^{n-1}_{r=1} {n+r-3\choose 2r-2}={n-2\choose 0}+{n-1\choose 2}+{n\choose 4}+\cdots+ {2n-4\choose 2n-4}&
\end{align}
is the difference $a_{n}-a_{n-1}$ \ref{differenceEvenFibo}. Since the difference of two consecutive even indexed Fibonacci number gives the odd indexed ones, the claim follows.
\end{proof}

\begin{example} Here we list algebras subject to the theorem above with a small number of simple modules $n$.
\begin{itemize}
\item $n=2$, $[2,1]$
\item $n=3$, $[3,2,1],[2,2,1]$
\item $n=4$, $[4,3,2,1], [3,3,2,1],[2,3,2,1],[3,2,2,1],[2,2,2,1]$
\end{itemize}
\end{example}

\section{Outlook on other classes of algebras}
In this article we mainly looked at the property of being S-connected for Nakayama algebras. But computer experiments with the GAP-package \cite{QPA} suggest that the property of being S-connected is also interesting for other classes of algebras such as blocks of Schur algebras. 
In fact all representation-finite or known tame blocks of Schur algebras are S-connected as well as some other classes.
For the definition of Schur algebras and the classification of representation-finite and tame types we refer for example to the survey article \cite{Kue}. We just illustrate here the example of the representation-finite blocks.
\begin{example}
Let $A=KQ/I$ be the algebra with the following quiver $Q$:
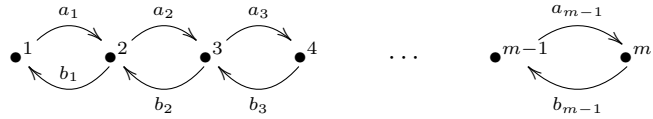
$$\xymatrix@1{\bullet^1 \ar@/^1pc/ [r]^{a_1} & \bullet^2 \ar@/^1pc/ [r]^{a_2} \ar@/^1pc/ [l]_{b_1}  & \bullet^3 \ar@/^1pc/ [r]^{a_3} \ar@/^1pc/ [l]^{b_2}& \bullet^4 \ar@/^1pc/ [l]^{b_3} & \cdots& \bullet^{m-1} \ar@/^1pc/ [r]^{a_{m-1}} & \bullet^{m} \ar@/^1pc/ [l]^{b_{m-1}}}$$
and with relations: $b_{m-1} a_{m-1}, \ b_{i-1}a_{i-1}-a_i b_i,\ a_{i-1}a_i, \ b_i b_{i-1}, \ i=2, \ldots, m-1$.
Any representation-finite block of a Schur algebra is Morita equivalent to such an algebra for some $m \geq 1$, see for example \cite[Theorem 3.12]{Kue}.
For a fixed number of simples $m$ we have have $pd(S_i)=m-2+i$ for $i=1,...,m$ and thus $A$ is S-connected.
We leave the elementary proof to the reader.

\end{example}
Our experiments lead to the following question:
\begin{question}
Let $A$ be a block of a Schur algebra. Is $A$ S-connected?

\end{question}
In fact, we are not aware of a quasi-hereditary connected algebra having a simple preserving duality which is not S-connected. This class of algebras includes all blocks of Schur algebras.

Finally we give an example of a quasi-hereditary algebra that is not S-connected:
\begin{example}
Let $A=KQ/I$ be the algebra with the following quiver:
$$
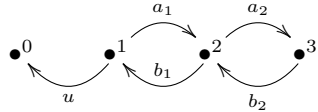
\xymatrix@1{\bullet^0 & \bullet^1 \ar@/^1pc/ [l]^{u} \ar@/^1pc/ [r]^{a_1} & \bullet^2 \ar@/^1pc/ [r]^{a_2} \ar@/^1pc/ [l]_{b_1}  & \bullet^3  \ar@/^1pc/ [l]^{b_2}& }$$
with the relations $a_1 a_2, b_2 a_2, b_2 b_1, a_2 b_2-b_1a_1$.
Basically $A$ results from glueing a sink to the quiver algebra of a representation-finite block of a Schur algebra with 3 simple modules.
The simple $A$-modules have projective dimension 0,2,3 and 4 and thus $A$ is not S-connected while $A$ is quasi-hereditary. We leave the easy proof again to the reader.
\end{example}

\section{Acknowledgements}
Rene Marczinzik is funded by the DFG with the project number 428999796.
We are thankful to Jeremy Rickard who allowed us to use his short proof of Proposition \ref{acyclicpropo}. The first named author thanks Dag Madsen for useful conversations on Nakayama algebras.
We profited from the use of the GAP-package \cite{QPA}.
Emre Sen is thankful to Gordana Todorov and Claus Michael Ringel for their help and support.

\end{document}